\documentclass[12pt,twoside]{amsart} 

\usepackage{amssymb,latexsym,bbm,graphicx,epsfig,epic,eepic,oldgerm,psfrag,comment}
\usepackage{amsmath,amsfonts,amscd,mathrsfs,amsthm}
\usepackage{mathtools}
\usepackage{mathtools}
\usepackage{forest}

\usepackage{tikz,wasysym}  \usetikzlibrary{trees}
\usepackage{marvosym}                  
\usepackage{hyperref} 

\usepackage{xypic} 
\input xy
\xyoption{all}

\theoremstyle{plain}
\newtheorem{theorem}{Theorem}[section]
\newtheorem{lemma}[theorem]{Lemma}                           
\newtheorem{proposition}[theorem]{Proposition}

\newtheorem*{remark*}{Remark}
\newtheorem*{remarks*}{Remarks}
\newtheorem{remark}[theorem]{Remark}

\newtheorem{example}[theorem]{Example}

\newtheorem*{example*}{Example}
\newtheorem*{examples*}{Examples}

\newtheorem*{definition*}{Definition}

\newtheorem{question}[theorem]{Question}

\newtheorem*{claim*}{Claim}

\numberwithin{figure}{section}
\numberwithin{equation}{section}

\newcommand{\proofend}{\hspace*{\fill} $\square$\\}

\def\1{\:\!}
\def\2{\;\!}

\def\Diffc0{\operatorname{Diff^c_0}}

\def\Sympc0{\operatorname{Symp^c_0}}

\def\CC{\mathbb{C}}

\def\QQ{\mathbb{Q}}
\def\RR{\mathbb{R}}

\def\ZZ{\mathbb{Z}}

\def\R{\operatorname{\mathbb{R}}}

\begin{document}

\title{\vspace*{0cm} Boundaries of open symplectic manifolds and the failure of  packing stability}
\author{Dan Cristofaro-Gardiner and Richard Hind}

\date{\today}

\maketitle

\begin{abstract}

A finite volume symplectic manifold is said to have ``packing stability" if the only obstruction to symplectically embedding sufficiently small balls is the volume obstruction.  Packing stability has been shown in a variety of cases and it has been conjectured that it always holds.  We give counterexamples to this conjecture; in fact, we give examples that cannot be fully packed by any domain with smooth boundary nor by any convex domain.  
The examples are symplectomorphic to open and bounded domains in $\mathbb{R}^4$, with the diffeomorphism type of a disc. 

The obstruction to packing stability is closely tied to another old question, which asks to what extent an open symplectic manifold has a well-defined boundary; it follows from our results that many examples cannot be symplectomorphic to the interior of a  compact symplectic manifold with smooth boundary.  Our results can be quantified in terms of the volume decay near the boundary, and we produce, for example, smooth toric domains that are only symplectomorphic to the interior of a compact domain if the boundary of this domain has inner Minkowski dimension arbitrarily close to $4$.  

The growth rate of the subleading asymptotics of the ECH spectrum plays a key role in our arguments.  We prove a very general ``fractal Weyl law", relating this growth rate to the Minkowski dimension; this formula is potentially of independent interest.

\end{abstract}

\section{Introduction}
\subsection{Context}

The existence of symplectic ball packings, that is, symplectic embeddings of disjoint balls, has been understood as a basic measure of symplectic flexibility at least since the work of Gromov, see the discussion in \cite[Sec. 1.1]{biran99}.
This flexibility can be quantified by the {\em packing numbers} $p_k(M, \omega)$ for $k \ge 1$. When $(M, \omega)$ has finite volume we define
$$p_k(M, \omega) := \sup_c \frac{k \mathrm{vol}(\overline{B(c)})}{\mathrm{vol}(M, \omega)}$$
where the supremum is over all $c$ such that there exists a  symplectic embedding $\sqcup \overline{B(c)} \hookrightarrow (M, \omega)$ of $k$ disjoint closed balls.

As explained in \cite[Ch. 6]{schl}, a theorem of McDuff and Polterovich \cite[Rem. 1.5.G]{mp1994} shows that
\begin{equation}\label{asymptote}
\lim_{k \to \infty} p_k(M, \omega) = 1
\end{equation}
for all finite volume symplectic manifolds. 
The volume filling limit is already a strong contrast with the Riemannian case.  For example, for isometric ball embeddings into domains in $\RR^2$, it has been known for ages that the limit of the analogous packing numbers is $\pi / \sqrt{12}$, see \cite{toth} for a rigorous proof; for domains in $\RR^3$, the corresponding limit 
%of the analogous packing numbers for isometric ball embeddings into domains in $\RR^3$ 
is $\pi / \sqrt{18}$ by celebrated work of Hales \cite{hales}; and, in any dimension, a general upper bound due to Rogers \cite{rogers} shows that the analogue of \eqref{asymptote} never holds.
%  , a bound due to Rogers'   
We say that a connected symplectic manifold has {\it packing stability}
if there exists a (stability number) $K_0$ such that $p_k(M, \omega) =1$ for all $k \ge K_0$.  
This represents perhaps a surprising level of flexibility, but packing stability has now been established for an ever increasing subset of symplectic manifolds. We say that a symplectic manifold $(M, \omega)$ is {\it rational} if $[\omega] \in H^2(M, \QQ) \subset H^2(M, \RR)$. Then the manifolds shown to have packing stability include closed rational 4-manifolds \cite{biran99}, closed rational manifolds of any dimension \cite{busehind11, busehind13}, all closed 4 manifolds \cite{busehindopshtein16}, open ellipsoids \cite{busehind13}, open 4 dimensional polydisks and pseudoballs \cite{busehindopshtein16}, and rational convex toric domains \cite{cgetal}. 

The question of determining which finite volume symplectic manifolds have packing stability was first raised in Biran's seminal work \cite{biran99}; it is featured as Problem 4 in \cite{chls}.
Based on the results above, and the asymptotic estimate \eqref{asymptote}, it was plausible to conjecture that 
the answer is simply that all do. 
Indeed, this is stated as Conjecture 13.2 in the survey article \cite{schlenk17}.

\subsection{Main results}\label{intro-main}

We now explain our main theorems.  In some cases, we can prove more general statements than these, see Remark~\ref{rmk:point}, but the theorems we present in this section get at the heart of our results while prioritizing brevity; we discuss how they fit together in Remark~\ref{rmk:point}.

The first point of the present work
is to give examples of symplectic manifolds which do not have packing stability; our examples are open subsets of $\mathbb{R}^4$. To state our result, recall that a {\it toric domain} $X \subset \RR^4 \equiv \CC^2$ is a subset of the form
$$X = X_{\Omega} := \{(z, w) \, | \, (\pi |z|^2, \pi |w|^2) \in \Omega\}$$
where $\Omega$ is a
subset of $\RR^2_{\ge 0}$. 
In this paper, we will be interested in 
toric domains $X_f := X_{\Omega}$, where $\Omega$ is the 
closed 
subset of the first quadrant bounded by the axes and a convex function $f:[0,\infty) \to (0,b]$ with $\int f(r) \, dr < \infty$; we call such a domain an {\em unbounded concave toric domain}.  We note that the function $f$ in the theorem does not have to be differentiable.

\begin{theorem} \label{thm:main}

Let $X_f$ be an unbounded concave toric domain.  Then 
$int(X_f)$ does not have packing stability. In fact, $p_k(int(X_f)) <1$ for all $k$.

\end{theorem}

Thus, the $X_f$ can not be fully filled, and so one would like to know what they look like.  The $X_f$ can be assumed to have smooth boundary, but since they are not compact this hides the behavior at infinity; we show below that they are symplectomorphic to interiors of compact domains, and it is more natural to study the boundary from this point of view.  %Now, for any symplectic set, the packing numbers depend only on the interior, and so 
Then understanding the nature of the closure of the image of $X_f$ %this 
connects to 
%Our next result is about 
a classic problem which goes back to \cite{eh}, of interest  beyond the context of packing stability: to what extent does an open symplectic manifold have a well-defined boundary, and what symplectic properties of the manifold are determined by the boundary?

In fact, we can give sharp results about the following facet of this.
In smooth topology, it is a fundamental problem to determine whether a given open manifold can be smoothly embedded as the interior of a smooth compact manifold with boundary: this is called the problem of ``finding a boundary", see for example \cite{bll}.  It is natural to ask the analogous question in the symplectic setting; prior to the present work, to our knowledge no symplectic obstructions to this problem have been known.  Returning now to our domains $X_f$, 
say that a function $f: [0,\infty) \to [0,b]$ has {\em asymptotic decay rate} $p \in \mathbb{R}$ if 
\[\lim_{x \to \infty} \frac{\ln f(x)}{\ln x} = -p.\]
For example, $f(x) = (1+x)^{-p}$ has asymptotic decay rate $p$: we denote the unbounded concave toric domain defined by this particular $f(x)$ by $X_p$.  We often call the asymptotic decay rate the {\em decay rate} for brevity.

\begin{theorem}
\label{main3}
Assume that $f$ has decay rate $1 < p < 2$, and assume that
$int(X_f)$ is symplectomorphic to
% the interior of 
a relatively compact subset $Z$ of a symplectic manifold.
% and assume that $f$ has decay rate $1 < p < 2$.
Then the inner Minkowski dimension of $\partial Z$ must be at least $2 + 2/p$.  In particular, $int(X_f)$ is not
symplectomorphic to the interior of any smooth compact symplectic manifold with piecewise smooth boundary. 
\end{theorem}

We recall the Minkowski dimension in section \ref{mink}.   The condition $p > 1$ is required for the volume to be finite, and in fact the above theorem is actually optimal, as the following proposition shows.

\begin{proposition}
\label{prop:minkowski}
For any $p > 1$, the manifold $int(X_p)$ is symplectomorphic to the interior of a compact set in $\CC^2$ with boundary of inner Minkowski dimension $\mathrm{max}(2+2/p, \, 3)$.
\end{proposition}

Note the 
relation with the seemingly mysterious condition $p < 2$: when $p > 2$, $int(X_f)$ is symplectomorphic to a bounded domain with inner Minkowski dimension $3$, and thus the above $2 + 2/p$ bound is no longer optimal.  We note that the Minkowski dimension is not in general a symplectic invariant: the domain $int(X_p)$ is also symplectomorphic to bounded sets of inner Minkowski dimension greater than $2 + 2p$.

The domains from Theorem~\ref{main3} give further examples of interesting packing stability phenomena, and this is the content of our final result.  To elaborate, let us now turn to packings by more general domains.  For example, the aforementioned Problem 4 in \cite{chls} asks for an investigation of packings by convex domains.  To our knowledge, not many results were previously known about packings by domains that are not symplectomorphic to toric domains, but our techniques apply well in this case. To state one result, let $(Z,\omega_1)$ be a symplectic manifold and define the {\em generalized packing number}
$$p_{Z,k}(M, \omega) := \sup_c \frac{k \mathrm{vol}(Z,c \cdot \omega_1)}{\mathrm{vol}(M, \omega),}$$
where the supremum is over all $c$ such that there exists a  symplectic embedding of $k$ disjoint copies of $(Z,c \cdot \omega_1)$.  
\begin{theorem}
\label{main2}
Let $(Z,\omega_1)$ be a symplectic manifold.  Assume that either:
\begin{enumerate}%[(i)]
\item $Z$ is compact with smooth boundary.
\item $Z$ is convex in $\CC^2$.
\end{enumerate}
Assume in addition that $f$ has asymptotic decay rate $1 < p < 2$.  Then $p_{Z,k}(int(X_f)) < 1$ for all $k$.
\end{theorem}

\begin{remark}
\label{rmk:super}
\normalfont
Theorem~\ref{thm:main} and Theorem~\ref{main2} have an interesting connection with the phenomenon known as ``long-term superrecurrence": this refers to situations where the images of any Darboux ball under iterates of a Hamiltonian diffeomorphism intersect more often than one would expect solely from volume considerations; for brevity, we refer the reader to the survey \cite{ps} for a fuller explanation. 
 To our knowledge, it had been unknown prior to the current work whether there are any natural symplectic manifolds $X$ for which long-term superrecurrence occurs.  It follows immediately from Theorem~\ref{thm:main} that this pheonmenon occurs whenever $X$ is an unbounded concave toric domain $X_f$.  Theorem~\ref{main2}  allows for a generalization of this, proving  
long-term superrecurence for a large class of open subsets of $X_f$.
\end{remark}

\begin{remark}
\label{rmk:point}
\normalfont
In combination Theorem~\ref{main2} and Theorem~\ref{main3} suggest the following perspective: the $X_f$ have provably complicated boundary, and this obstructs any full filling by a domain with qualitatively simpler boundary.   A more general statement, which implies both of the aforementioned theorems, appears in Theorem~\ref{thm:general}.  We further discuss this perspective in section \ref{mainproofs}.   
\end{remark}

\begin{remark}
\normalfont
As observed by S. Nemirovski \cite{nemerovski}, the domains $X_p$ give examples of the following phenomenon: they are logarithmically convex and hence holomorphically convex, but are not biholomorphic to a bounded domain (as they contain copies of the complex plane $\CC$); on the other hand, in the symplectic category, it follows from our results that they are symplectomorphic to bounded domains but not to convex domains. 
\end{remark}

\subsection{The nature of the obstructions and the fractal Weyl law}

The main idea in proving the above theorems is to study the ``subleading asymptotics" of ECH capacities\footnote{One could alternatively use the subleading asymptotics of the ``elementary ECH capacities" defined in \cite{elem}, and our proofs would hold verbatim for these as well, since they only rely on formal properties that are shared by both sets of capacities.} $c_k$.
% share; 
These are symplectic capacities, defined for any subset of a symplectic $4$-manifold; we summarize what we need to know about ECH capacities in section \ref{ECHprelim}.  The leading asymptotics of these capacities were studied in \cite{cghr}, with the conclusion (under some hypotheses that are not relevant to our discussion here, but which hold, for example, for any bounded open subset of $\mathbb{R}^4$) that
\[ c_k(X) \approx \sqrt{4 vol(X)} \sqrt{k}.\]
In other words, the $c_k$ grow like $k^{1/2}$, and the coefficient at highest asymptotics measures the volume of the manifold.

Much less is currently understood about the subleading asymptotics
\[ e_k(X) := c_k(X) - 2\sqrt{vol(X) k}.\]
The $e_k$ are the invariants that we use to obstruct packing stability, and another consequence of our work is a better understanding of the $e_k$ themselves. More precisely, 
we can extract well-defined growth rates taking a wide range of values for bounded open connected subsets of $\mathbb{R}^4$, contrasting the case of the leading asymptotics. 
To state our result, fix $2 \ge p > 1$, and recall $X_p$ is the toric domain $X_f$ with $f = (1+x)^{-p}$. 

\begin{proposition} 
\label{prop:sub}
There exist positive constants $C_1$ and $C_2$ so that 
\[ -C_1 k^{1/2p} < e_k(X_p) < -C_2 k^{1/2p}.\]
\end{proposition}

In view of the above proposition, it is natural to ask 
%It is an interesting question
 to what degree the growth rate of the $e_k$ sees information about a general finite volume domain $U$.  In fact, we can prove the following.
 % ``fractal Weyl law", which is reminiscent of the kind of fractal Weyl laws that abound in other areas of geometry (see e.g. [ref]). 
 % which resembles  
 Define the {\em ECH dimension}
% \begin{equation}
 %\label{eqn:dech}
 \[d_{ECH} := 2 + 4 \limsup_{k \to \infty} \left(\frac{\ln(-e'_k(Z))}{\ln(k)}\right)\]
 %\end{equation}
 where $e'_k(Z) = min( e_k(Z), 0)$, and let $dim_M$ denote the inner Minkowski dimension.

\begin{theorem}
\label{thm:fractal}
Let $Z$ be a relatively compact and open subset of a symplectic four-manifold.  Then
\[ d_{ECH}(Z) \le dim_M(\partial Z).\]  
\end{theorem}

%The formula in Theorem~\ref{thm:fractal} is reminiscent of results which abound 
We can think of this as a  ``fractal Weyl law", bounding the Minkowski dimension from below by a symplectic invariant, in analogy with formally similar fractal Weyl laws which abound in other areas of geometry, see e.g. \cite{weylannals, dd, berry, physics, mv} and the references therein.

\subsection{Outline of the paper and the main proofs}  We now explain the structure of the paper.

Section \ref{prelims} reviews what we need to know about the 
volume growth near the boundary of Riemannian manifolds, the Minkowski dimension, and ECH capacities, and collects some key preliminary results.

The next two sections are about ECH capacities.  Theorems \ref{thm:main} and \ref{main2} can be read as obstructions to symplectic embeddings between domains of equal volume, where, by the definition of packing numbers, we allow small rescalings of the domain. Since the ECH capacities of domains of equal volume have the same leading asymptotics, the subleading asymptotics can give obstructions.  Section \ref{lowerbounds} establishes general lower bounds on the subleading asymptotics 
for any symplectic manifold, via the aforementioned fractal Weyl law,
extending and quantifying an argument of Hutchings \cite{h2}. Section \ref{upperbounds} establishes upper bounds for unbounded concave toric domains.

Based on this, we establish our main proofs in section \ref{mainproofs}. Concave toric domains are shown to be symplectomorphic to bounded subsets of $\CC^2$ in section \ref{sec:bound}.  The final section discusses packing stability in view of our new results.

\vspace{0.1in}

{\bf Acknowledgements.} We would like to thank Dusa McDuff, Felix Schlenk and Richard Schwartz for reading a preliminary version of the text and providing very helpful comments.   DCG also thanks the National Science Foundation for their support under agreement DMS-2227372, and RH thanks the Simons Foundation for their support under grant no. 633715.  Key conversations about this project occurred during a visit by RH to the University of Maryland, hosted by the Brin Mathematics Research Center;  we thank the center for their support.  We also received valuable comments from the audience on an earlier version of this work that was presented by DCG at the ``From smooth to $C^0$ symplectic geometry" workshop at CIRM in July, 2023; we are very grateful to the organizers and to CIRM for their hospitality.

\section{Preliminaries}\label{prelims}

\subsection{Volume decay and the Minkowski dimension}\label{mink}

We begin by collecting some facts about the Minkowski dimension.  

Our first order of business in this section involves clarifying some definitions.  Let $U \subset \mathbb{R}^4$ be open, and assume that $\partial U$ is compact.  Recall the {\em Minkowski dimension} of $\partial U$, which can be defined as follows \cite{tao}.  Let $\hat{V_d}(U)$ be the volume of $\lbrace x \in \mathbb{R}^4 \, | \, \mathrm{dist}(x,\partial U)  < d \rbrace$.  We now define the {\em upper Minkowski dimension} $dim_M(\partial U)$ by the formula:
\[ \mathrm{dim}_M(\partial U) = 4 - \liminf_{d \to 0} \frac{ \ln \hat{V_d}(U) }{\ln(d)}.\]
(There is also a lower Minkowski dimension, but this is not relevant to our results.)
The Minkowski dimension depends only on $\partial U$, and we will be interested in a slight refinement that is sensitive to $U$ itself.  Namely, define  $V_d(U)$ to be the volume of $\lbrace x \in U \, | \,\mathrm{dist}(x,\partial U)  < d \rbrace$.  We define the {\em inner Minkowski dimension} in analogy by the formula
\begin{equation}
\label{eqn:inner}
\mathrm{dim}_M(\partial X) = 4 - \liminf_{d \to 0} \frac{ \ln V_d(U) }{\ln(d)}.
\end{equation}
It is immediate from the definition that the inner Minkowski dimension bounds the upper Minkowski dimension from below. It is clear that the inner Minkowski dimension is equivalent to the quantity
\begin{equation}
\label{eqn:decay}
\liminf_{d \to 0} \frac{ \ln V_d(U) }{\ln(d)}.
\end{equation}
which we call the {\em order of the volume decay rate near $\partial U$}.

We will be interested in subsets of a general symplectic manifold $(M,\omega)$ -- this is the natural setting for Theorem \ref{main2}, for example --- so we will want to generalize the above definitions.   This is straightforward as long as $U \subset M$ is contained in a compact set; recall that such a set is called {\em relatively compact}.  Let $g$ be a Riemannian metric on $M$ and $U \subset M$ relatively compact and open; define $V_{d,g}(U)$ by copying the definition in the Euclidean case from above, but replacing the Euclidean metric with the metric $g$.
The following lemma states that the volume decay and thus the inner Minkowski dimension is independent of the choice of metric.

\begin{lemma}\label{metricdep}
Let $Z$ be an open and relatively compact subset of a smooth manifold $M$.  Let $g$ and $g'$ be two Riemannian metrics on $M$.
Then
\[ \liminf_{d \to 0} \frac{ V_{d,g}(U) }{d^q}  > 0\]
if and only if
\[ \liminf_{d \to 0} \frac{ V_{d, g'}(U)}{d^q}  > 0. \]
\end{lemma}

\begin{proof}
Let $K$ be a compact set containing $Z$; let $\omega$ and $\omega'$ be the volume forms for $g$ and $g'$ respectively; let $h$ and $h'$ be the corresponding distance functions, mapping $M \times M$ to $\mathbb{R}$.
Then $\omega = f \omega'$ for some smooth positive function $f: M \to \mathbb{R}$ and similarly, $h = p h'$ for some smooth positive function $p: M \times M \to \mathbb{R}$.  The function $f$  (resp. $p$) has strictly positive upper and lower bounds on $K$ (resp. on $K \times K$).  
Since by definition $\lbrace x \in Z | \mathrm{dist}_g (x,\partial Z) < d \rbrace \subset K$, with the same fact true for $g'$, the lemma follows.
\proofend
 \end{proof}

Thus, we can define the inner Minkowski dimension and the volume decay for a bounded open subset of a general manifold $M$ by choosing a metric $g$ on $M$, and copying the definitions in \eqref{eqn:inner} and \eqref{eqn:decay}, replacing the quantity $V_d$ with $V_{d,g}$.

The next lemma proves what we need to know about the volume decay near piecewise smooth sets.  

\begin{lemma}
\label{lem:smoothcase}
Let $X \subset \RR^4$ be a bounded domain whose boundary is 
a finite union of compact subsets of smooth codimension $1$ submanifolds of $\RR^4$. Then there is a constant $C$ with $V_d(X) < Cd$ for all $d$.
\end{lemma}

\begin{proof} Write the boundary as a union of compact subsets $C_i$ 
of smooth submanifolds $S_i$.  Then points in within $d$ of the boundary of $X$ lie in the image of one of the exponential maps
$${\mathrm exp}_{S_i} : \nu^d S_i|_{C_i} \to \RR^4$$
where ${\mathrm exp}_{S_i}$ is the exponential map of $S_i$ and $\nu^d S_i$ is the (rank $1$) normal disk bundle of radius $d$. As the $S_i$ have codimension $1$ these sets are each 
contained in
sets having volume of order $d$.
\proofend
\end{proof}

There is a similar result when the boundary is convex, even if it is not piecewise smooth.

\begin{lemma}
\label{lem:convexcase}
Let $Z \subset \mathbb{R}^4$ be an open and bounded convex domain.  Then $\partial Z$ has volume decay of order $1$.   
\end{lemma} 

\begin{proof}  The proof is a simple application of Steiner's formula, which says in the case relevant here that for any convex body $Z \subset \mathbb{R}^4$, the volume $N_r(Z)$ 
is given by the polynomial
\[ N_r(Z) = \sum^4_{i=0} Q_i(Z)r^i,\]
where the $Q_i(Z)$ are the ``quermassintegrals".
Now consider $Z(d) = \lbrace z \in Z \, | \, \mathrm{dist}(x,\partial Z)  \ge d \rbrace$.  This is a convex domain, and since $Z$ can be alternatively described as the radius $d$ tube around $Z(d)$, the volume of $\lbrace z \in Z \, | \, \mathrm{dist}(x,\partial Z)  < d \rbrace$ is given by
\[  N_d(Z(d)) - N_0(Z(d)).\]
It is well-known that the quermassintegrals are Hausdorff continuous and positive, thus the lemma follows.
\proofend
\end{proof}

\subsection{ECH capacities and concave toric domains}\label{ECHprelim}

The proof of Theorem~\ref{thm:main} uses the subleading asymptotics of the ECH capacities $c_k$, applied to a particular limit of concave toric domains.  In this section, we briefly recall what we need to know about ECH capacities and concave toric domains.

The ECH capacities are a sequence of nonnegative real numbers $c_k(M,\omega)$ associated to four-dimensional symplectic manifolds; they were defined in \cite{h1}.   One can think of ECH capacities as measurements of symplectic size.   A key axiom that they satisfy is the Monotonicity Axiom, which states that each $c_k$ is monotone under symplectic embeddings.  We also need to recall the Ball Axiom, which states that $c_k(B(a))$ is the $(k+1)^{st}$ smallest entry in the matrix $(ma+na)_{ (m,n) \in \mathbb{Z}_{\ge 0} \times \mathbb{Z}_{\ge 0}}$, and the Disjoint Union axiom, which states that $c_k(X \sqcup Y) = sup_{i+j = k} (c_i(X) + c_j(Y))$ for compact four-dimensional Liouville domains $X$ and $Y$.  The definition of ECH capacities arises from the ``embedded contact homology" of closed three-manifolds with contact forms, and ECH capacities are  most naturally defined for compact four-dimensional Liouville domains.  However, in this paper we will need to consider slightly more wild objects.  What we require for our purposes is that, as observed by Hutchings, the definition of ECH capacities extends to any open subset $M$ of a four-dimensional symplectic manifold by defining 
\[  c_k(M) := \sup_{Z} \hspace{1 mm} c_k(Z),\]
where $Z$ is a compact four-dimensional Liouville domain in $M$, with many of the same properties holding.  For more, we refer the reader to \cite{h1}.

Another important axiom about ECH capacities is the Volume Axiom, which states that for compact Liouville domains the $c_k$ recover the Volume in their limit; see \cite{cghr} for the precise formula.  Our obstruction to packing stability will come from the subleading asymptotics, calculated for the domains in Theorem~\ref{thm:main}.   These domains are limits of concave toric domains and we now recall what we need to know about such domains from \cite{ccghr}.  A {\em concave toric domain} is a toric domain $X_{\Omega}$, where $\Omega$ is the subset of the first quadrant bounded by the axes and a convex function $f:[0,a] \to [0,b]$ with $f(a) = 0$. (Recall in section \ref{intro-main} we defined an unbounded concave toric domain to be defined by a convex function $[0, \infty) \to (0, b]$.)  For example, a ball is a concave toric domain.  Any concave toric domain $X_{\Omega}$ has a canonical (possibly infinite) packing by balls $B(a_i)$, via the procedure explained in \cite[Sec. 1.3]{ccghr}; we call the $a_i$ the {\em weights} associated to $\Omega$.   To help the reader visualize the decomposition, we note that this packing comes from decomposing the region $\Omega$ into a union of triangles, each of which are equivalent under the action of the affine group to an isoceles right triangle; for more, see \cite[Fig. 1]{ccghr}. 

An important fact about concave toric domains \cite{ccghr} is that the weights determine the ECH capacities, via the formula:
\begin{equation}
\label{eqn:concaveball}
c_k(X_\Omega) = c_k( \sqcup_i B(a_i)).
\end{equation}      
This formula is proved for concave toric domains such that the weight expansion is finite in \cite[Thm. 1.4]{ccghr}, and then \cite[Rem. 1.6]{ccghr} explains how to extend this to arbitrary concave toric domains; only the finite case is needed as an input for this paper.
We will also need the following bound which immediately follows via induction from the definition of the weight sequence; it also follows as a special case of the more involved \cite[Lem. 3.6]{h2}.

\begin{lemma}
\label{lem:bound}

Let $X_{\Omega}$ be a concave toric domain, such that the associated weights satisfy $\sum a_i \le M.$
Then $X_{\Omega} \subset \lbrace  \pi |z_1|^2 \le M \rbrace \subset \mathbb{C}^2.$
\end{lemma}

\section{The subleading asymptotics of ECH capacities and the Minkowski dimension}\label{lowerbounds}

We now begin the work towards the proofs of the main results.
Our main theorems all come from considerations of the subleading asymptotics of the ECH capacities. 
The aim of this section is to prove the following general theorem, which is equivalent to the fractal Weyl law Theorem~\ref{thm:fractal}, but stated in a form that is more useful for our purposes.

%of potentially independent interest.

\begin{theorem}\label{bdrygrowth} Let $Z$ be a relatively compact and open
subset of a symplectic manifold $(M,\omega)$.  Suppose that the volume decay near $\partial Z$ is of order $q$.  Then there exists a positive constant $C$, depending only on $Z$, such that
$e_k(Z) > - C k^{(2-q)/4}$
for all $k$.
\end{theorem}

One does not expect a similar upper bound in this level of generality; for example, even for smooth compact Liouville domains the ECH capacities could in principle be infinite, see for example the discussion after \cite[Thm. 1.1]{cghr}.

The proof makes use of the following proposition, which is a quantified variant
of Theorem 1.1 from \cite{h2}. 

\begin{proposition} \label{hest} Let $X \subset \RR^4$ be an open
 domain 
and suppose there are constants $C$ and $q$ such that $V_d(X) < Cd^q$.  Then there exists a positive constant $C$, depending only on $X$, such that
 $e_k(Z) \ge - C k^{(2-q)/4}$ for all $k$.
\end{proposition}

\begin{proof} We follow closely the argument in \cite[Thm. 1.1]{h2}, which essentially deals with the case $q=1$.

Let $P(a,a) = X_{\Omega}$ be a polydisc: this corresponds to the case where $\Omega$ is a square of side lengths $a$ with legs on the axes.  One first estimates the subleading asymptotics $e_k$ of a polydisc by using the fact that the ECH capacities in this case have an explicit combinatorial formula, see \cite[Lem. 4.1]{h2}, with the conclusion that $e_k(P(a,a)) \ge -2a$.  The polydiscs $P(a,a)$ are symplectomorphic to cubes.

We now fill the interior of $X$ with open cubes, by making use of the lattice $2^{-n} \mathbb{Z}^4$, as follows.   First, we add in all open cubes contained in $X$ with vertices on neighboring $2^{-1} \mathbb{Z}^4$ lattice points.  Next, we add in all open cubes contained in $X$ with vertices on neighboring $2^{-2} \mathbb{Z}^4$ lattice points that do not intersect the cubes from the previous step.  We continue inductively in this way, getting at the $n^{th}$ step an embedding into $X$ of a disjoint union of a finite number of (open) cubes, each of volume $2^{-4n}$.

Now we apply \cite[Lemma 4.2]{h2}, which says the following. (The lemma was originally stated for bounded domains, but this assumption is not used in the proof.)

\begin{lemma} Suppose there exists a symplectic embedding $$\sqcup P_i \hookrightarrow X$$ where $P_i$ is a copy of $P(a_i, a_i)$. Let $I_k = \{ i \, | \, a_i^2 \ge \mathrm{vol}(X)/k \}$ and let $$V_k = \mathrm{vol} \left( \cup_{i \in I_k} P_i \right) = \sum_{i \in I_k} a_i^2.$$
Then $$e_k(X) \ge -2 \sqrt{2} \sum_{i \in I_k} a_i + 2 \frac{ (V_k - \mathrm{vol}(X) )}{\sqrt{\mathrm{vol}(X)}} \sqrt{k}.$$
\end{lemma}

Returning to our situation, let $X_n$ be the union of cubes constructed in the first $n$ steps. We see that $X \setminus X_n \subset N_{2^{1-n}}$ and so has volume at most $C 2^{-qn}$. As the cubes added in the $n^{th}$ step each have volume $2^{-4n}$, we see the number added is bounded above by $C 2^{n(4-q)}$.

Now suppose $$16^n \le \frac{k}{\mathrm{vol}(X)} < 16^{n+1}.$$

Then $I_k$ is exactly the indices corresponding to the cubes added in the first $n$ steps.

We have $$\sum_{i \in I_k} a_i \le \sum_{j=1}^n C 2^{j(4-q)} 2^{-2j} < C 2^{n(2-q)}$$
and $$V_k - \mathrm{vol}(X) > - C 2^{-qn}.$$
Therefore the lemma gives $$e_k(X) \ge  - C k^{(2-q)/4} $$ as required.

\proofend
\end{proof}

We can now prove the main theorem from this section.

\begin{proof}({\em Proof of Theorem~\ref{bdrygrowth}})

Since $Z$ is relatively compact, it is contained in a compact set $K$.  Since $K$ is compact, we can cover it with finitely many compact Darboux balls $B_1, \ldots, B_n$, with domain centered at the origin in $\mathbb{R}^4$; we identify  these balls with their image.  The union of the $B_i$ admits a 
refinement into a collection of compact sets $S_j$, $1 \le j \le N$, 
with disjoint interiors: let $A$ be the union of the $B_i$, let $A'$ be the set obtained by subtracting all of the boundaries of the $B_i$ from $A$, and then take the $S_i$ to be the closures of the connected components of $A'.$
(Here, by a refinement we mean that each $S_j$ is contained in a $B_i$, and the union of the $S_j$ is precisely the union of the $B_i$.)
We identify each $S_j$ with its image in $M$.

Now consider $\Delta_j = int(Z \cap S_j)$.  This is symplectomorphic to an open and bounded subset $Q_i$ of $(\mathbb{R}^4,\omega_{std})$.  The identification of $S_j$ with a subset of $M$ induces a decomposition
\[ \partial Q_i = X_{reg} \cup X_Z ,\]
where $X_{reg}$ are those points that correspond to points on the boundary of some $S_j$, and $X_Z$ are those points that correspond to points in $\partial Z \subset M$.  Letting $N_d(V) \subset V$ be the points within distance $d$ of $\partial V$, we therefore have
\[ N_d(\partial Q_i) = N_d( X_{reg} ) \cup N_d(X_Z).\]
The set $X_{reg}$ is contained in a finite union of compact hypersurfaces, since its image in $M$ is contained in the union of the boundaries of the $S_i$.  Thus, by Lemma \ref{lem:smoothcase} the volume of $N_d( X_{reg} )$ is $O(d)$.  On the other hand, by the argument in Lemma \ref{metricdep}, the volume of $N_d(X_Z)$ is, up to a constant, bounded by the volume of the tube $N_d(\partial Z)$ in $M$ and thus $O(d^q)$ by assumption.

Now, we have $$c_k(Z) \ge c_k( \sqcup \Delta_j)$$ and we estimate the right hand side using the Disjoint Union Axiom.

\[ c_k( \sqcup \Delta_j) = \sup_{i_1 + \dots + i_N = k} \sum c_{i_j}(\Delta_j) \]
By choosing $i_j = \lfloor k \frac{ vol(\Delta_j)}{vol(Z)} \rfloor$, Lemma \ref{hest} gives
\[ \sup_{i_1 + \dots + i_N = k} \sum c_{i_j}(\Delta_j) \ge \sum 2 \sqrt{ k \frac{ vol(\Delta_j)^2}{vol(Z)} - \epsilon_j vol(\Delta_j) } - C_j i_j^{(2-q)/4},\]
where $0 \le \epsilon_j < 1$ is the fractional part of $k\frac{ vol(\Delta_j)}{vol(Z)}$.

Continuing, the right hand side in the above inequality is bounded from below by
\[ \sum\left( 2 \sqrt{ k \frac{ vol(\Delta_j)^2}{vol(Z)}} - 2\sqrt{\epsilon_j vol(\Delta_j) } - C_j i_j^{(2-q)/4}\right) \]
\[ = 2 \sqrt{ k vol(Z) } - \sum \left( 2\sqrt{\epsilon_j vol(\Delta_j) } + C_j i_j^{(2-q)/4} \right)\]
\[ \ge 2 \sqrt{ k vol(Z) } - 2N vol(Z) - C k^{(2-q)/4}, \]
for $C = \sum C_j$, since $$\sum C_j i_j^{(2-q)/4} \le (\sum C_j) max_j(i_j)^{(2-q)/4} \le (\sum C_j) ( \sum i_j)^{(2-q)/4}.$$
The required bound for $e_k(Z)$ follows immediately.
\proofend

\end{proof}

\section{The subleading asymptotics of unbounded concave toric domains}\label{upperbounds}

The previous section gave general lower bounds on the subleading asymptotics of any relatively compact open subset.
In this section, we aim to prove the results we will need about unbounded concave toric domains, which will be an essentially precise calculation of the subleading asymptotics.  More precisely, we aim to prove the following:

\begin{proposition}\label{upperprop}
Let $X$ be an unbounded concave toric domain. Then $$\lim_{k \to \infty} e_k(\text{int}(X))  = - \infty.$$
In the case when $X$ is defined by a convex function $f$ with decay rate $1 < p \le 2$, for any $\epsilon>0$
there exists a constants $C_1$ and $C_2$ such that $$-C_1 k^{1/2p + \epsilon} < e_k(X_p) < -C_2 k^{1/2p - \epsilon}.$$
\end{proposition}

\begin{remark}
\label{rmk:upperp}
The proof of the upper bound in the above proposition does not use the condition $p < 2$ at all; however, we do not need this result for any of the applications in our paper.
\end{remark}

\subsection{The weight expansion and the ECH capacities}

We begin by generalizing the relationship between the ECH capacities and the weight sequence for concave toric domains to the case of unbounded concave toric domains.

To shorten the notation, denote any unbounded concave toric domain $X_{\Omega}$
by $X.$
To prove the theorem, it is helpful to approximate $X$ by concave toric domains $X_n$, as follows.  The definition of the weight sequence in \cite[Sec. 1.3]{ccghr} extends verbatim to $X$, even though it is infinite, since the function $f$ is convex.  Thus we obtain a sequence of weights $a_i$ associated to $X$, which also have the property that 
\[ a_1, \ldots, a_n\]
are the weights for some concave toric domain; we denote this concave toric domain by $X_n$.   There is some flexibility in how to order the $a_i$; we assume that they are in nonincreasing order.

Note that by definition,
\begin{equation}
\label{eqn:exhau}
 X_1 \subset X_2 \ldots \subset X_n \subset \ldots
 \end{equation}
The following lemma explains what we need to know regarding how the domains $X_n$ approximate $X$:

\begin{lemma}
\label{lem:exhaustion}
$int(X) = \cup_n int(X_n)$ 
\end{lemma}

\begin{proof}
An equivalent description of $\cup_n int(X_n)$, that follows from the definition of the weight expansion, is that  $\cup_n int(X_n) = X_{\Omega'}$, where $\Omega'$ is the union of all triangles
bounded by the axes and a\footnote{Recall that by a {\em tangent line} in this possibly non-smooth, but convex, setting, we mean a straight line through a point on the graph of $f$ whose intersection with the first quadrant stays entirely in the region bounded by the axes and the graph of $f$.} tangent line to the graph of $f$ with rational slope, with the vertical and horizontal axes contained in $\Omega'$ but not the hypotenuse.    

It follows from the definitions that $X_{\Omega'} \subset int(X)$.  To see the opposite inclusion, given any point $x \in int(X_{\Omega})$, let $(a,b)$ denote the corresponding point in $\Omega$, and choose a point $(c,d)$ on the graph of $f$ with $c > a$ and $d > b$.  Then $(a,b)$ lies in the triangle bounded by the axes and any tangent line to the graph of $f$ at $(c,d)$, and is not on the hypotenuse.  If there is a tangent line at $(c,d)$ with rational slope then the corresponding triangle lies in $\Omega'$ and we are done; if it does not, we can 
find a point $(c', d')$ on the graph of $f$ which has a tangent line of rational slope sufficiently close to the irrational slope at $(c,d)$, and the corresponding triangle will still contain $(a,b)$. 
\proofend
\end{proof} 

We can now extend \eqref{eqn:concaveball} to our manifold $X$, which is a limit of concave toric domains.

\begin{lemma}
\label{lem:key}

For every $k$
\[ c_k(int(X)) = c_k( \sqcup_i B(a_i) ) .\]

\end{lemma}

\begin{proof}  By the Monotonicity Property of ECH capacities \cite{h1}, we have
\[ c_k(int(X)) \ge c_k( \sqcup_i B(a_i)),\]
so we have to establish the opposite inequality.   By definition \cite[Defn. 4.9]{h1}, 
\[ c_k(int(X)) := \sup_{Z} \hspace{1 mm} c_k(Z),\]
where $Z$ is a compact Liouville domain in $int(X)$. 
Then $Z$ embeds into some $int(X_n)$.  To see this, assume the opposite, and define a sequence $z_n \in Z$ such that $z_n$ is not in $int(X_n)$; pass to a subsequence converging to some $z_{\infty} \in Z$; the point $z_\infty$ is in $int(X)$, hence by Lemma~\ref{lem:exhaustion} it is in $int(X_n)$ for some $n$, which is a contradiction in view of \eqref{eqn:exhau}.   

Thus we have $c_k(Z) \le c_k(X_n)$ for all $k$.  On the other hand, by \eqref{eqn:concaveball}, $c_k(X_n) = c_k( \sqcup^n_{i=1} B(a_i))$ and so we have that
\[ c_k(Z) \le c_k(\sqcup_i B(a_i)).\]
Therefore,
\[  \sup_{Z} \hspace{1 mm} c_k(Z) \le c_k(\sqcup_i B(a_i)),\]
and so the lemma is proved.
\proofend
\end{proof}

The following is a useful method to estimate weights of concave toric domains, including the unbounded case.

\begin{lemma} \label{weightestimate} Suppose $f'(x) = -1/i$. Then
$a_i \ge f(x)$.
\end{lemma}

\begin{proof} It suffices to find $i$ weights $a_j$ all satisfying $a_j \ge f(x)$.  It follows from the definition of the weight expansion that for each positive integer $n$, there is a weight for the associated toric domain corresponding to an affine triangle whose slant edge meets the graph of $f$ at the point $(x,y)$ such that $f'(x) = -1/n;$ moreover, it follows from the definition of the weight expansion that this weight is at least as large as $y$. The result follows.
\proofend

\end{proof}

\subsection{Estimating the subleading asymptotics in terms of the weights}

We can now prove the first part of Proposition \ref{upperprop}; we also a prove a useful upper bound for the subleading asymptotics in terms of the weights.
Throughout $C>0$ will denote a, possibly varying, positive constant, which is independent of $k$.

\begin{proposition}
\label{prop:main}
If $\sum a_i$ is unbounded, then
\begin{equation}
\label{eqn:growth}
\lim_{k \to \infty} e_k(\mathrm{int}(X))  = - \infty.
\end{equation}
Moreover if $a_i \ge  C \left(\frac{1}{i} \right)^{\frac{p}{p+1}}$, then  
\begin{equation}
\label{eqn:infsub}
e_k(\mathrm{int}(X))    \le - Ck^{\frac{1}{2p}}.
\end{equation}
\end{proposition}

\begin{proof}  By Lemma~\ref{lem:key}, it suffices to prove the same result for $c_k( \sqcup_i B(a_i))$.

Let $d(k) = \lbrace d(k)_i \rbrace_{i=1,\ldots}$ be a maximizer over the nonnegative integers $d_i$, $i \ge 1$, of the optimization problem
\begin{equation}
\label{eqn:opt}
\left \lbrace \sum_i d_i a_i \hspace{1 mm} | \hspace{1 mm} \sum_i (d_i^2 + d_i) \le 2k \right \rbrace,
\end{equation}
noting that such a maximum does exist and can be assumed to have the properties that $d(k)_j$ is nonincreasing in $j$, $d(k)_j = 0$ for $j > k$, and
\begin{equation}
\label{eqn:max}
\sum_i d(k)_i^2 + d(k)_i = 2k.
\end{equation}

The reason we are interested in the $d(k)_i$ is that by the Disjoint Union \cite[Prop. 1.5]{h1} and Ball \cite[Cor. 1.3]{h1} axioms reviewed above, we have that
\begin{equation}
\label{eqn:opt}
\sum_i a_i d(k)_i = c_k(M);
\end{equation}
see also \cite[Sec. 1.1]{h2} for further discussion of the infinite case.

Now
\[ \sum_i ( d(k)^2_i + d(k)_i) = 2k.\]
Applying Cauchy-Schwarz we obtain
\[ \sum_i a_i \sqrt{d(k)^2_i + d(k)_i} \le \sqrt{ 2 vol} \sqrt{2k}\]
where $\mathrm{vol} := \mathrm{vol}(X) = \frac{1}{2} \sum a_i^2$.

By invoking \eqref{eqn:opt} and Lemma \ref{lem:key}, we can rewrite the above inequality as
\[ e_k(\mathrm{int}(X)) = c_k( \sqcup B(a_i) ) - 2 \sqrt{k \mathrm{vol}} \le - \sum_i a_i \left( \sqrt{ d(k)^2_i + d(k)_i} - d(k)_i \right).\]

Let $I(k)$ be the minimum $i$ such that $d(k)_i = 0$.   The function $n \to \sqrt{ n^2 + n}  - n$ is a nonnegative, strictly increasing function on the domain $\mathbb{Z}_{\ge 0}$, 
so by above, 
\[ e_k(\mathrm{int}(X)) \le - (\sqrt{2}-1) \sum^{I(k)-1}_i a_i. \]
Thus, by Lemma \ref{lem:bound}, to show \eqref{eqn:growth}, it suffices to show that 
\begin{equation}
\label{eqn:needed}
\lim_{k \to \infty} I(k) = +\infty.
\end{equation}
As for \eqref{eqn:infsub}, if we assume that 
$a_i \ge C \left(\frac{1}{i} \right)^{\frac{p}{p+1}}$,
then we have
\[ - (\sqrt{2}-1) \sum^{I(k)-1}_i a_i \le -C \int^{I(k)}_1 \frac{1}{x^{p/(p+1)}} dx\]
\[ = - C (I(k)^{\frac{1}{p+1}} -1),\]
so that it suffices to prove that 
\begin{equation}
\label{eqn:needed2}
I(k) \ge C k^{\frac{p+1}{2p}}.
\end{equation}

We find the needed lower bounds on $I(k)$ as follows.  Let $j < I(k)$.  Then we claim that
\begin{equation}
\label{eqn:bound1}
-a_j + a_{I(k)} + \ldots + a_{I(k)+d(k)_j - 1} \le 0.
\end{equation} 
Indeed, if we consider a sequence of integers $d'(k)_i$ (indexed by $i$) that is identical to $d(k)_i$  except that 
\[ d'(k)_j = d(k)_j - 1, \quad \quad d'(k)_{I(k)} = \ldots = d'(k)_{I(k) + d(k)_j - 1} = 1,\] 
then the $d'$ satisfy $\sum_i (d(k)_i'^2 + d(k)_i') = \sum_i (d_i(k)^2 + d(k)_i)$, and the claimed inequality follows immediately from the fact that the $d(k)_i$ are maximizers of the optimization problem \eqref{eqn:opt}.  

Now \eqref{eqn:bound1}, applied to $j = 1$, implies \eqref{eqn:needed}.  Otherwise, there is a subsequence with $I(k)$ is uniformly bounded, and necessarily on this subsequence $d(k)_1$ is unbounded. But then the divergence of $\sum a_i$ gives a contradiction (as $a_1$ is finite).

Thus, it remains to prove \eqref{eqn:needed} under the assumption that $a_i \ge C \left(\frac{1}{i} \right)^{\frac{p}{p+1}}$, and this will be our standing assumption for the rest of the proof.

By \eqref{eqn:bound1} and the fact that the $d(k_j)$ are nonincreasing, we have that $-a_j + d(k)_j a_{I + d(k)_j - 1} \le 0$, and so
\begin{equation}
\label{eqn:bound2}
d(k)_j \le \frac{a_j}{a_{I(k)+d(k)_j - 1}} \le \frac{a_j}{a_{I(k) + \lceil \sqrt{2k} \rceil}},
\end{equation}
where in the last inequality we have used the fact that $d(k)_j \le \sqrt{2k} + 1$. 

Hence, in view of the inequality above, the identity \eqref{eqn:max}, and the definition of $I(k)$, we have:
\[ 2k = \sum^{I(k)}_{j=1} d(k)^2_j + d(k)_j \le \sum^{I(k)}_{j=1} 2 d(k)^2_j \le 2 \sum^{I(k)}_{j=1} \frac{a^2_j}{ a^2_{I(k) + \lceil \sqrt{2k} \rceil} } \le \frac{4 vol}{  a^2_{I(k) + \lceil \sqrt{2k} \rceil}}.\]
It follows that
\[ a^2_{I(k) + \lceil \sqrt{2k} \rceil} \le \frac{4 \mathrm{vol}}{2k}\]
hence
\[ ( I(k) + \sqrt{2k} + 1)^{2p/(p+1)} \ge C k.\]

The inequality \eqref{eqn:needed2} follows from this since $1/2 < (p+1)/(2p)$.

\proofend
\end{proof}

\subsection{Subleading asymptotics and decay rates}

In this section we obtain tight bounds for the subleading asymptotics of concave domains in terms of their decay rates, and state precise bounds in the case of the $X_p$, see Proposition \ref{ekfordecay2}. These bounds are an application of the lower bounds coming from section \ref{lowerbounds} and the upper bounds from section \ref{upperbounds}.

Recall we suppose $f$ is convex with decay rate $1 < p \le 2$, that is, $$\lim_{x \to \infty} \frac{\ln f(x)}{\ln x} = -p.$$

We will prove the following.

\begin{proposition}\label{ekfordecay} For all $\epsilon>0$ sufficiently small there exists a constants $C_1$ and $C_2$ such that $$-C_1 k^{1/2p + \epsilon} < e_k(X_f) < -C_2 k^{1/2p - \epsilon}.$$
\end{proposition}

First we estimate the volume growth. 

\begin{lemma}\label{volest2}  For all $\epsilon >0$ there exists a constant $C$ such that $V_d(X_f) < Cd^{2 - 2/p - \epsilon}$.
\end{lemma}

\begin{proof} 
Given $\epsilon>0$, for $x$ sufficiently large we have $f(x) < x^{-p+ \epsilon}$ and so there is a constant with $f(x) < C(1+x)^{-p+ \epsilon}$ for all $x$.

We think of $X_f$ as a fibration over the $z$ plane with fibers over $z$ being disks of area $f(\pi |z|^2) < C(1 + \pi|z|^2)^{-p + \epsilon}$.

Define $N_d$ as before to be the points within $d$ of the boundary and $W_d \subset N_d \subset X_p$ to be the points within $d$ of the boundary of the corresponding fiber. There is a constant $C$ (depending on the slope of $f$) such that $N_d \subset W_{Cd}$, and hence it suffices to bound the volume of the $W_d$.

The complement of $W_d$ is another toric domain, now corresponding to the region under the graph of
$$g(r) = \pi \max(0, \sqrt{ f(r) / \pi} - d)^2.$$
Let $f(r_{\infty}) = \pi d^2$, so $r_{\infty} < C d^{-2/( p - \epsilon)}$ where we again use $C$ to denote a constant, in this case independent of $d$.

Therefore we have
\[ \mathrm{vol}(W_d) = \int_0^{r_{\infty}} (f(r) - \pi ( \sqrt{f(r) / \pi} - d)^2 ) \, dr + \int_{r_{\infty}}^{\infty} f(r) \, dr \]
\[ = \int_0^{r_{\infty}} (2\sqrt{\pi}d \sqrt{f} - \pi d^2) \, dr  + \int_{r_{\infty}}^{\infty} f(r) \, dr \]
\[ <   Cd (1+x)^{1-(p-\epsilon)/2} \big|_0^{r_{\infty}}  - C (1+x)^{1-p+\epsilon} \big|_{r_{\infty}}^{\infty} < C d^{2 - 2/p - \epsilon} \]
since $p \le 2$, as required. (We note that the same proof applies when $p>2$ but then the integral of $\sqrt{f}$ converges and we obtain growth of order $d$.)
\proofend
\end{proof}

\begin{lemma}\label{derivest} Let $\epsilon >0$. Then for $n$ sufficiently large we have $|f'( n^{1/(p+1) + \epsilon})| \le 1/n$.
\end{lemma}

\begin{proof} We argue by contradiction, and assume $|f'(x)| > 1/n$ on the interval $[n^{1/(p+1)}, n^{1/(p+1) + \epsilon} ]$.

Then the change $\Delta f$  in $f$ over the interval is bounded by
\begin{equation}
\label{eqn:change}
\Delta f > \frac{1}{n} (n^{1/(p+1) + \epsilon} - n^{1/(p+1)}) = n^{\frac{-p}{p+1}}( n^{\epsilon} -1).
\end{equation}

On the other hand, the decay rate of $f$ implies that for large $x$ we have
$$x^{-p- \epsilon} < f(x) < x^{-p + \epsilon}$$
\begin{equation}
\label{eqn:change2}
\Delta f < n^{1/(p+1)(-p + \epsilon)} - n^{(1/(p+1) + \epsilon)(-p - \epsilon)} \\
= n^{\frac{-p}{p+1}}( n^{\epsilon / (p+1)} - n^{ -\epsilon( p + 1/(p+1) + \epsilon)}).
\end{equation}

This gives a contradiction since when $n$ is large we have $$n^{\epsilon / (p+1)} - n^{ -\epsilon( p + 1/(p+1) + \epsilon)} < n^{\epsilon} -1.$$

\proofend
\end{proof}

We can now prove Proposition \ref{ekfordecay}.

\begin{proof}
First, Lemma \ref{volest2}  together with Proposition \ref{hest} implies the lower bound on $e_k(X_f)$. Second, Lemma \ref{derivest} together with Lemma \ref{weightestimate} implies that the $n$th term of the weight expansion for $X_f$ is at least $$f(n^{1/(p+1) + \epsilon}) \ge n^{(1/(p+1) + \epsilon)(-p-\epsilon)} = n^{-\frac{p}{p+1} - \epsilon(p + \frac{1}{p+1}) - \epsilon^2}.$$
Then, scaling $\epsilon$ appropriately, Proposition \ref{prop:main} gives the result.

\proofend
\end{proof}

In the case of the domains $X_p$ the defining function is equal to $(1+x)^{-p}$. We can follow the same argument to estimate the subleading asymptotics but now without any order $\epsilon$ error terms. This gives the following stronger result, which was previously stated as Proposition~\ref{prop:sub}; we recall the statement:

\begin{proposition} \label{ekfordecay2} There exist positive constants $C_1$ and $C_2$ so that $-C_1 k^{1/2p} < e_k(X_p) < -C_2 k^{1/2p}$.
\end{proposition}

\section{A quantified theorem and proofs of the main results}\label{mainproofs}

\subsection{The main theorems}

We can now prove our main results.  We begin with the failure of packing stability for all unbounded concave toric domains.

\vspace{1 mm}

\begin{proof}{\em (Proof of Theorem~\ref{thm:main})}

Let $n \ge 1$ and $D_n$ denote the disjoint union of $n$ equal open balls, with total volume $\mathrm{vol} = \mathrm{vol}(\text{int}(X))$.  

By \cite[Lem. 2.4]{mcd}, we have
\begin{equation}
\label{eqn1}
 c_k (D_n) = c_k ( \lambda E(1,n ) ),
 \end{equation}
for $\lambda = \sqrt{\frac{ 2 \mathrm{vol} } { n} }$.
Moreover, by \cite[Prop. 15]{cgs},
\begin{equation}
\label{eqn2}
e_k(E(1,n)) > - \infty.
\end{equation}

Hence, by \eqref{eqn1} together with \eqref{eqn2} and Proposition~\ref{upperprop},
there exist $k$ such that
\[  c_k(D_n) = c_k ( \lambda E(1,n ) ) > c_k(\text{int}(X)).\]

As the ECH capacities of $D_n$ vary continuously under rescaling the balls, we also get a strict inequality for ECH capacities of sufficiently large compact subsets of $D_n$. Hence by the Monotonicity property of ECH capacities, \cite[Prop. 4.11]{h1}, we see that such compact subsets of $D_n$ do not embed in the interior of $X$.
\proofend

\end{proof}

\begin{remark}(Ellipsoid embedding functions)
\normalfont
A series of works, see e.g. \cite{mcdsch, cv, cgetal, special, cgfs, cgk, chls, siegel1, siegel2} and the references therein, study the ``ellipsoid embedding function" $c_X(a)$ of a symplectic manifold; in the case where $(X,\omega)$ is a $4$-manifold, this is defined to be the the supremum, over $c$, such that the symplectic ellipsoid $c \cdot E(1,a)$ can be symplectically embedded into $(X,\omega)$.  In this literature, one often distinguishes in the finite volume case between the ``flexible case", where the only embedding obstruction is the volume, and the ``rigid case", where there are further obstructions.  In all examples studied prior to our work, it has been the case that flexibility holds whenever $a$ is sufficiently large.  The argument for Theorem~\ref{thm:main} above implies that flexibility for the ellipsoid embedding function never holds for any $a$ whenever $(X,\omega)$ is an unbounded concave toric domain; alternatively, this follows directly from Theorem~\ref{thm:main}, since it is known that packing stability by balls holds for any four-dimensional symplectic ellipsoid \cite{busehind13}.

\end{remark}

Our other two main results, Theorem~\ref{main2} and Theorem~\ref{main3},
are 
corollaries of the following more general theorem, which is of potentially independent interest.  The theorem is a kind of quantified version of our obstruction to packing stability, in terms of the Minkowski dimension.  To state the theorem, say that a symplectic manifold $(Z,\omega)$ {\em fully packs} $(M,\omega')$ if the generalized packing number $p_{Z,1}(M,\omega') = 1.$

\begin{theorem}
\label{thm:general}
Let $f$ have decay rate $1 < p < 2$ and let $(Y,\omega)$ be a either a relatively compact and open subset of a symplectic manifold, or an open subset of $\RR^4$. 
If $(Y,\omega)$ fully packs $X_f$, then the inner Minkowski dimension of $Y$ must be at least $2 + 2/p$.
\end{theorem}

We recall from section \ref{mink} that the inner Minkowski dimension is well defined in these situations.

Before proving Theorem~\ref{thm:general}, we explain why it implies the remaining main results.

\vspace{1 mm}

\begin{proof}{\em (Proof of Theorem~\ref{main2}, assuming Theorem~\ref{thm:general})}
Let $(Z,\omega_1)$ and $k$ be as in the statement of that theorem and let $Y_k$ be the disjoint union of $k$ copies of $(Z,\omega_1)$.  Then
\begin{equation}
\label{eqn:packingnumberk}
 p_{Y_k,1}(X_f) = p_{Z,k}(X_f).
 \end{equation}
On the other hand, the inner Minkowski dimension of $(Z,\omega_1)$ is $3$: in the case where $Z$ is compact with smooth boundary, this follows from Lemma~\ref{lem:smoothcase}, and in the case where $Z$ is a convex domain this follows from Lemma~\ref{lem:convexcase}.  Thus the inner Minkowski dimension of $Y_k$ is $3$, and so by Theorem~\ref{thm:general} and \eqref{eqn:packingnumberk} 
\[ p_{Z,k}(X_f) < 1\] 
hence Theorem~\ref{main2}.  
\proofend
\end{proof}

\vspace{1 mm}

\begin{proof}{\em (Proof of Theorem~\ref{main3}, assuming Theorem~\ref{thm:general})}

Assume that $int(X_f)$ is symplectomorphic to $int(Z)$.  Then, $int(Z)$ fully packs $int(X_f)$.  Thus by Theorem~\ref{thm:general} the inner Minkowski dimension of $Z$ is at least $2 + 2/p$.
\proofend
\end{proof}

\vspace{2 mm}

We now give the promised proof of our quantified variant of the failure of packing stability.

\begin{proof}{\em (Proof of Theorem~\ref{thm:general})}

Let $q$ be the order of the volume decay near $\partial Y$, so that the inner Minkowski dimension of $Y$ is $4 - q$, and let $c$ be such that $(Y, c \omega)$ has the same volume as $X_f$.
Then by Theorem~\ref{bdrygrowth}, or by Proposition \ref{hest} in the Euclidean case,
\begin{equation}
\label{eqn:bound1} 
e_k(Y, c \omega) > - C k^{(2-q)/4}.
\end{equation}
On the other hand, by Proposition~\ref{upperprop} 
\begin{equation}
\label{eqn:bound2}
e_k(X_p) < -C_1 k^{1/2p - \epsilon},
\end{equation}
with the constant $C_1$ depending on $\epsilon > 0$.  Thus if 
\[ 4 - q < 2 + 2/p\]
then for $\epsilon$ sufficiently small 
\[ -C_1 k ^{1/2p - \epsilon} <  - C k^{(2-q)/4} \]
for $k$ sufficiently large, hence, in combination with \eqref{eqn:bound1} and \eqref{eqn:bound2}, 
\[ c_k(Y, c \omega)  > c_k(X_p)    \] 
for  some $k$.
Thus, by the Scaling and Monotonicity properties,  $(Y, c' \omega)$ does not embed into $X_p$ for any scaling factor $c'$ such that the  volume is sufficiently close to the volume of $X_p$, and so the theorem is proved. 

\end{proof}

\subsection{Some further applications around Theorem~\ref{thm:general}}

Theorem~\ref{thm:general} is more general than our main theorems, so we now give some examples of some further applications beyond our main results.

\begin{example} (Infinite packings) 
\normalfont
Theorem~\ref{thm:general} does not require the domain to have even finitely many components.  For example, we can study infinite ball packings.  
The $X_f$, being unbounded concave toric domains, have a canonical ball-packing; however, we can obstruct many other ball-packings.  For infinite ball packings, the inner Minkowski dimension is easy to compute, and in general we find that for balls $\sqcup_{i \ge 1} B(a_i)$ to fully fill an $int(X_f)$ with decay rate $1 < p < 2$, the partial sums $\sum^k_{i=1} a_i$ must grow sufficiently fast. To give just one example, let
$Y = \sqcup_{i \ge 1} B(i^{-q})$ with $1/2 < q < 2/3$.  One can compute that this has inner Minkowski dimension $2/q$ and so the theorem implies the disjoint union fully fills an $X_f$ with decay rate $p$ only if $q < \frac{p}{p+1}$.  
\end{example}

\begin{example} (Unbounded concave domains with different decay rates)
\normalfont
Let $X_{f_1}$ and $X_{f_2}$ be unbounded concave toric domains, such that the decay rate of $f_1$ is $p$, the decay rate of $f_2$ is $q$, and $1 < p < q < 2.$  Then, it follows from 
that $int(X_{f_2})$ does not fully fill $int(X_{f_1})$, since
 $int(X_{f_2})$ has 
 inner Minkowski dimension $2 + 2/q$, and this is smaller than $2 + 2/p$; alternatively, we can apply Proposition \ref{ekfordecay}.   It would be interesting to know if it must be the case that $int(X_{f_1})$ can fully fill $int(X_{f_2})$
\end{example}

\section{Boundedness.}
\label{sec:bound}

In this section, we address Proposition \ref{prop:minkowski}. More generally, we show that a large class of toric domains $X_{\Omega}$ (including concave toric domains) of finite volume are symplectomorphic to the interior of compact subsets of $\CC^2$. At least in the case of the $X_p$ we show the boundary of these compact sets have the same inner Minkowski dimension as the $X_p \subset \CC^2$ themselves.

We use coordinates $z,w$ on $\CC^2 \equiv \R^4$ and write $z = x+iy, w = u + iv$. The round disk of area $R$ is denoted by $D(R)$.  

\begin{proposition}
\label{prop:bound} Suppose $f$ is a nonincreasing function with $f(0)=1$ and $V = \int_0^{\infty} f(r) \, dr < \infty$. Let $\Omega = \{s \le f(r) \}$ and $X = X_{\Omega} = \{(z,w) | (\pi|z|^2, \pi|w|^2) \in \Omega\}$.

For any $\delta >0$ there exists a symplectic embedding $$F: X \hookrightarrow D(2+ \delta) \times D(V+2).$$
Moreover:
\begin{enumerate}
\item $F(X)$ is the interior of the compact set $\overline{F(X)}$.
\item In the case when $X = X_p$, the volume decay rate of 
%$Y$ 
$F(X)$ is equal to that of $X$.
\end{enumerate}

\end{proposition}

\begin{remark} 
\normalfont
We note there exist toric domains $X \subset \CC^2$ of finite volume but which do not admit symplectic embeddings into any bounded set. For example we can take $X = X_{\Omega}$ where $\Omega$ is an open neighborhood of the diagonal $\{r=s\}$. The first examples of this nature were due to Hermann, \cite{hermann}.
\end{remark}

\begin{proof} The proof adapts  the symplectic folding construction to a form more sensitive to the Minkowski dimension.
The construction of the embedding has 3 steps. In each step we define a symplectic embedding, $\Phi$, $\Psi$ and $\Xi$ respectively, and their images $X_1 = \Phi(X)$, $X_2 = \Psi(X_1)$ and our bounded domain $X_3 = \Xi(X_2)$. The construction will show that $\Psi$ and $\Xi$ have bounded derivatives (unlike $\Phi$). Hence to estimate the inner Minkowski dimension it suffices to look at $X_1$, and so we describe this set in detail.

\vspace{0.1in}

{\em Step 1.  Preparation. The embedding $\Phi$.}  

Our goal here is to map our manifold $X$, which is a fibration over the $z$ plane $\CC$ with fibers $D(f(\pi|z|^2))$, to a fibration over a strip $T$.

We start with the region $$T = (0, \infty) \times (0,1)$$ in the first quadrant of the $z$ plane.
Let $\phi_1$ be a symplectic embedding of the $z$ plane into itself with image $T$.
We can choose the embedding such that all disks $D(R)$ map into a neighborhood of the region $\{ x < R \} \cap T$.
Here we note that the regions $\{ x < R  \} \cap T$ have area $R$ and so such an embedding is possible by \cite{schl}, Lemma 3.1.5. We arrange that when $S>R$ the image $\phi_1(D(S) \setminus D(R))$ intersects $\{ x < R  \} \cap T$ only in a very small neighborhood of its boundary.

Let $\phi_2$ be a symplectomorphism of the $w$ plane mapping the disks $D(R)$ to approximations of squares $S(R)$, of side length $\sqrt{R}$, centered at the origin. Again we can construct such a symplectomorphism using \cite{schl}, Lemma 3.1.5.

Now we define $\Phi = \phi_1 \times \phi_2$. We think of the image $X_1 = \Phi(X)$ as fibered over $T$. We claim that the set is an approximation (in a sense we describe momentarily) of the set $\tilde{X_1}$ which is the fibration over $T$ with fibers $S(f(x))$.

The set $\tilde{X_1}$ then is a manifold with corners, and $X_1$ itself will have smooth boundary, smoothly approximating the smooth parts of $\partial \tilde{X_1}$ and rounding the corners.

To see the approximation, suppose $x = R$ for some $R>0$ and $(x,y) \in \phi_1(\CC)$ is away from the boundary . Then $(x,y)$ is disjoint from $\phi_1(D(S))$ when $S$ is slightly smaller than $R$, but if $(x,y)$ lies away from $\partial \phi_1(\CC)$ then it also avoids the images of points with $\pi |z|^2$ greater than $R$. It follows that $(x,y) = \phi_1(z)$ where $\pi |z|^2$ is approximately $R$, and so the fiber must approximate $\phi_2(D(f(R)))$ or $S(f(R))$.
We can choose this approximation as close as necessary.

Points $(x,y)$ with $x = R $ and close to $\partial \phi_1(\CC)$ lie in the image of a point with $\pi|z|^2 >R$, and we can arrange $\pi|z|^2$ to increase rapidly, and hence the size of the fiber to decrease rapidly, as we approach the boundary.

\vspace{0.1in}

{\em Step 2.  Displacing fibers. The embedding $\Psi$.}

Our goal is to describe a domain $X_2 = \Psi(X_1)$ with the property that the $w$ fibers over points $(x_1, y_1)$, $(x_2, y_2)$ in the $z$ plane with $|x_1 - x_2| \ge 1$ are disjoint. We will do this by translating the squares $S(f(x))$ in the $u$ direction, as we describe now.

\vspace{0.1in}

{\em Step 2a.  Moving the squares.}
Let $$G(w) = v + 1/2.$$
Then as $f(x) \le 1$ we have $0 \le G \le 1$ on $X_1$ and the Hamiltonian vector field $X_G = \frac{ \partial}{\partial u}.$ The flow is time independent and displaces a square $S(R)$ in time $\sqrt{R}$.

\vspace{0.1in}

{\em Step 2b.  Displacing the fibers.}
Fix a  function $\chi : [0, \infty) \to \RR$ such that $\chi(x) = x$ when $x \le 1$ and $\chi'(x)= \sqrt{ f(x-1)}$ when $x \ge 1$. We recall our normalization $f(0)=1$.

Then we consider the Hamiltonian diffeomorphism $\Psi$ generated by $$H(z,w) = \chi(x) G(w).$$

To study the effect of $\Psi$ we apply the product rule to write
$$X_{H}(x, y, w) = \chi(x)X_{G}(w) + G(w) \chi'(x)\frac{ \partial}{\partial y}.$$
We observe that $\Psi$ preserves the $x$ coordinate, and as $G$ is an integral of the motion the $\frac{ \partial}{\partial y}$ component is independent of time. This implies that $\Psi$ can be written as a composition $\Psi^z \circ \Psi^w$, where $\Psi^z$ fixes the $w$ coordinate and translates in the $y$ direction by $G(w) \chi'(x)$, while $\Psi^w$ fixes the $z$ coordinate and flows along $X_G$ for time $ \chi(x)$.

We have
\begin{equation}\label{zmove}
\Phi^z(x,y,u,v) = (x, y + (v+ 1/2)\chi'(x), u, v).
\end{equation}
In particular, as $0 \le G = v+ 1/2  \le 1$ on $X_1$, we see that the translation $\Psi^z$ is nonnegative and bounded by $\chi'(x) \le 1 + \delta$. Hence the projection of $X_2= \Psi(X_1)$ to the $z$ plane now lies in a strip
$$T_2 = (0, \infty) \times (0,2 + \delta).$$

The diffeomorphism 
\begin{equation}\label{wmove}
\Psi^w (x,y,u,v) = (x,y, u + \chi(x), v).
\end{equation}
Like $X_1$ the image $X^w_2 = \Psi^w(X_1)$ also fibers over $T$, but now the fiber over a point $(x,y)$ is the square $S(f(x)) + \chi(x)e_u$ where $e_u = (0,0,1,0)$ is the unit vector in the $u$ direction.
Our condition on $\chi'$ implies that $\chi(x) - \chi(x-1) > \sqrt{ f(x-1)}$, which is the width of the fiber over points $(x-1,y)$, and hence the fibers of $X^w_2$ over points $(x_1,y_1)$ are now disjoint from those over points $(x_2, y_2)$ with $|x_1 - x_2| \ge 1$. These fibers remain disjoint if we apply $\Psi^z$ and hence we also have that the fibers of $X_2 = \Psi(X_1)$ over points $(x_1,y_1)$ and $(x_2, y_2)$ with $|x_1 - x_2| \ge 1$ are disjoint. 
It is clear that both $\Psi^z_k$ and $\Psi^w_k$ have uniformly bounded derivatives.

In preparation for the next step, we compute the area of the projection $S$ of $X_2^w$ (or equivalently $X_2$) onto the $w$ plane. For this, if $(u,v)$ is the projection of $\Psi^w(x,y,u',v)$ then, as $(u', v) \in S(f(1)) = S(1)$ and hence $u' \in [-1/2, 1/2]$, formula \ref{wmove} gives $\chi(x) \in [u-1/2, u+1/2]$. Therefore, if $u \ge 1/2$, the $v$ coordinate lies in an interval of length at most $\sqrt{ f( \chi^{-1}(u-1/2))}$. When $u \in [-1/2, 1/2]$ we still have $v$ in an interval of length $1$. Hence the area of $S$ is bounded by
\begin{equation}
\label{eqn:bound3}
\begin{split}
1 + \int_{1/2} ^{\infty} \sqrt{ f( \chi^{-1}(u-\frac{1}{2}))} \, du  & = 1 + \int_0^{\infty} \sqrt{f(y)} \chi'(y) \, dy \\
& < 2 + \int_1^{\infty} f(y-1) \, dy \le V+2.
\end{split}
\end{equation}
where we use the identity $\chi'(y) = \sqrt{ f(y-1) }$ when $y \ge 1$ and recall $f(y) \le 1$ for all $y$.

\vspace{0.1in}

{\em Step 3.  Wrapping.}   To complete the description of our embedding let $\xi_1$ be  an  immersion of the region $T_2 = (0, \infty) \times (0, 2+\delta)$ in the $z$-plane with image contained in $D(2+2\delta) \setminus \{0\}$ and having the property that $\mathrm{arg} \xi_1(z)$ depends only upon $x$, that is, vertical line segments map to rays. We arrange that points $(x_1,y_1)$ and $(x_2, y_2)$ map to the same ray if and only if $x_1 - x_2 \in \ZZ$. Concretely, we can choose polar coordinates $R = \pi |z|^2$, $\theta \in \RR / \ZZ$ on the plane, and set $\xi_1(x,y) = (R, \theta)$ where $R = y + \delta$ and $\theta = -x$.

Let $\xi_2$ be a symplectomorphism of the $w$ plane mapping the region $S$ into the disk $D(V+2)$. This is possible by the calculation at the end of Step 2, and for the purposes of the current paper our concern is just that the disk has finite area. The map can be chosen to wind the infinite strip $S$ continuously around the origin, mapping points with $u$ large close to the boundary of the disk.

To be concrete, we recall that $S$ lies in the region $$\left[\frac{-1}{2}, \frac{1}{2} \right] \times \left[\frac{-1}{2}, \frac{1}{2} \right] \bigcup \{ u \ge 1/2, \, |v|<  \sqrt{ f( \chi^{-1}(u-1/2))}/2 \}.$$ Let $\sigma : [-1/2, \infty) \to \RR$ satisfy $\sigma(u) = u + 1/2 + \delta$ when $u \le 1/2$ and $\sigma'(u) = \sqrt{ f( \chi^{-1}(u-1/2))}$ when $u \ge 1/2$. Then using polar coordinates as before we can define $\xi_2(u,v) = (R, \theta)$ with $R = \sigma(u) + v$, $\theta = -u$. This gives a symplectic embedding of $S$ into a disk of area $$\lim_{u \to \infty} \sigma(u) = \delta + 1 + \int_{1/2}^{\infty} \sigma'(u) \, du < V+2$$ by the calculation at the end of Step 2.

We see that both $\xi_1$ and $\xi_2$  have bounded derivatives. We define $\Xi  = \xi_1 \times \xi_2$ and set $X_3 = \Xi(X_2)$.

{\em Step 4. Injectivity.}  We note that any double points $(x,y)$, $(x', y')$ of $\xi_1$ necessarily have $x-x'  \in \ZZ \setminus \{ 0 \}$, and hence the fibers of $X_2$ over such points are disjoint.  Therefore the map $\Xi$ is an embedding of $X_2$ and $F = \Xi \circ \Psi \circ \Phi$ gives an embedding of $X$, with image lying inside $D(2 + 2\delta) \times D(V+2)$. In particular $X$ is symplectomorphic to a bounded domain.

{\em Step 5.  Interior of a compact set.}  Having described our embedding $F$, and proved part of the proposition, we now prove the rest.  For the second statement in Proposition \ref{prop:bound} we study the bounded open set $Y = X_3 = F(X)$ and claim that $Y$ is the interior of $\overline{Y}$. To see this, we need to show that points in $\partial Y$ do not lie in the interior.

In fact we can describe $\partial Y$ explicitly: it is a union of two sets $Z_1 \cup Z_2$ we define now.

The embedding $\Xi \circ \Psi$ extends to an embedding of an open neighborhood of $X_1 \subset \CC^2$, and we set $Z_1 = \Xi \circ \Psi(\partial X_1)$. We recall $X_1$ is unbounded but arranged to have smooth boundary. As $\Xi \circ \Psi$ extends smoothly across $\partial X_1$, points in $Z_1$ are also in the boundary of $\overline{Y}$.

The remainder $Z_2$ of $\partial Y$ consists of accumulation points $(z_{\infty},w_{\infty})$ of sequences $\Xi \circ \Psi(z_i, w_i)$, where $(z_i, w_i) \in X_1$ and $\mathrm{Re}z_i \to \infty$. By construction, such points have $\pi |w_{\infty}|^2 = \lim_{u \to \infty} \sigma(u)$ where $\sigma(u)$ is the map described in Step 3. This limit is a finite number $l < V+2$, and as any neighborhood of $(z_{\infty},w_{\infty})$ contains points with $\pi |w|^2 > l$, we see that $Z_2$ is also contained in the boundary of $\overline{Y}$ as required. 

\vspace{3 mm}

{\bf Remark.} As an aside we note (at least in the case when the integral of $\sqrt{f}$ diverges, that is, when $\chi(x)$ is unbounded) that the set $Z_2$ is actually precisely the product $$Z_2 = \{ \delta \le \pi |z|^2 \le 1 + \delta\} \times \{ \pi |w|^2 = l\}.$$ Indeed, when $x = \mathrm{Re}z$ is large, $\chi'(x)$ is small and the map $\Psi^z$ from formula \ref{zmove} is small. Hence if $x$ is large the $w$ fiber of $X_2 = \Psi(X_1)$ over a point $(x,y)$ is a square $S(f(x)) + \chi(x)e_u$ when $0 < y < 1$ and empty otherwise (up to small perturbations which decrease with $x$). In particular, applying $\Xi$ to $X_2 \cap \{x > K\}$ when $K$ is large, the $z$ projection of the image approximates the image of the map $\xi_1$ applied to $\{x > K, \, 0<y<1\}$, which is the annulus $ \{ \delta < \pi |z|^2 < 1 + \delta\}$. Thus $Z_2$ lies in the set described.

Conversely, fix $0 \le r <1$, $0<y<1$ and consider the sequence of points $z_n = n+ r + iy$ for $n \in \mathbb{N}$. Under the map $\xi_1$ from Step 3, all $z_n$ are mapped to $\xi_1(z_n) = \xi_1(r + iy) = (R= y + \delta, \theta = -r)$.  Acting on the fiber of $X_2$ over $z_n$, which is a square centered at $(\chi(n+r),0)$, the argument of $\xi_2$ forms an interval around $\theta = -\chi(n+r)$. As $\chi$ is unbounded but has derivative decreasing to $0$, these points form a dense set of arguments as $n$ increases, showing that all points in the circle $\{ \pi |w|^2 = l\}$ appear as accumulation points of the fiber of $Y$ over $\xi_1(r + iy)$.

\vspace{3 mm}

{\em Step 6.  Minkowski dimension.}  To conclude the proof of Proposition \ref{prop:bound} we carry out some estimates in the case when $X = X_p$, that is, when $f(x) = (1+x)^{-p}$. We focus on the case when $1<p<2$, though the same argument works when $p \ge 2$.  

As the maps $\Psi$ and $\Xi$ have uniformly bounded derivatives, the Minkowski dimension of $F(X)$ will be equal to that of $X_1 = \Phi (X)$. The nature of our approximation implies that this is also the Minkowski dimension of $\tilde{X_1}$.

Let $N_d$ be the points within $d$ of the boundary of $\tilde{X_1}$. This contains all square fibers of side length less than $2d$, that is, the fibers over points with $x > x_{\infty}$ where $\sqrt{ f(x_{\infty})} = 2d$ or $x_{\infty} = Cd^{-2/p}$ (throughout $C$ denotes a constant independent of $d$). The volume of the total space of these fibers is then
$$ \int_{x_{\infty}}^{\infty} f(x) \, dx  = C d^{2 - 2/p}.$$

Next we consider the fibers over points with $x < x_{\infty}$. The set $N_d$ is now contained in a union of four subsets: 
\begin{enumerate}
\item fibers over points with $x < d$; 
\item fibers over points with $y<d$; 
\item fibers over points with $y>1-d$; 
\item points $(x,y,w)$ with $w \in S(f(x)) \setminus S(g(x))$, where $S(g(x))$ is a square of side length $\sqrt{f(x)} - 2d$.
\end{enumerate}

Set (1) has volume bounded by $df(0) = d$. Sets (2) and (3) have volume bounded by $$d \int_0^{x_{\infty}} f(x) \, dx < Cd.$$ Finally set (4) has volume bounded by
 $$\int_0^{x_{\infty}} (f(x) - g(x)) \, dx = \int_0^{x_{\infty}} (2d \sqrt{f(x)} - 4d^2) < C d^{2 - 2/p}$$ using $g(x) = (\sqrt{f(x)} - 2d)^2$.

Putting the estimates together gives our Minkowski dimension as required.

\proofend

\end{proof}

\section{Discussion}
\label{sec:disc}

Given our Theorem~\ref{thm:main}, and the context surveyed in the introduction, it is natural to now revisit the question of which finite volume symplectic manifolds have packing stability.  
In view of the positive results for closed manifolds surveyed in the introduction, the natural guess would be that all closed manifolds do.
On the other hand, as we have shown here, packing stability can fail for non-compact manifolds even though it sometimes holds in this class.  One would like to get a better handle on the interface between these phenomena: a natural question is what properties of a non-compact manifold guarantee that packing stability holds.

In the cases presented here, the boundary behavior certainly appears to be playing a key role.
Whether the singular character of the boundary in our examples is necessary for packing stability to fail seems a natural starting point for further study.   Here is a sample question:

\begin{question}
Must packing stability hold for every compact Liouville domain with smooth boundary? 
\end{question}

All known non-compact examples for which packing stability holds have quite a lot more structure than this.  In this regard, let us also note that our obstruction here comes from the divergence of the subleading asymptotics of ECH capacities for our domains.   However, various conjectures imply that in sufficiently nice cases these should converge; see e.g. \cite[Conj. 1.5]{h2}.  

In a different direction, one would like to know if analogues of our results holds in higher dimensions.  The perspective conveyed in Remark \ref{rmk:point} certainly seems like it should hold in arbitrary dimension, but the obstructions we use here come from ECH capacities, which do not currently have a good generalization in higher dimensions.
Thus the following two basic questions remain open:

\begin{question}
Can packing stability fail in dimension $ > 4$?
\end{question}

\begin{question}
Are there ``boundaryless" discs in $\mathbb{R}^{2n}$ when $n > 2$?  More precisely, do there exist bounded and open subsets of $\mathbb{R}^{2n}$, diffeomorphic to discs, that can not be symplectomorphic to the interior of any compact symplectic manifold?
\end{question}

It would also be interesting if one can improve on the condition $p < 2$ in Theorem~\ref{main2}, for example with respect to packings by smooth domains: this would follow immediately from our results if one had better lower bounds on the growth rate of the subleading asymptotics of ECH capacities of smooth compact manifolds, which is an interesting problem of independent interest.

\end{document}